\documentclass[12pt,leqno]{article}
\usepackage{amscd,amsfonts,amsmath,amssymb,amsthm,amstext,latexsym}
\usepackage{enumerate,pifont,epsf,graphicx}
\usepackage[english]{babel}
\usepackage[T1]{fontenc}
\usepackage{epsfig}
\usepackage[latin1]{inputenc}

\def\a{{\alpha}}
\def\g{{\gamma}}
\def\b{{\beta}}
\def\O{{\Omega}}
\newcommand{\R}{{\mathbb R}}
\newcommand{\N}{{\mathbb N}}
\newcommand{\htwo}{{\mathbb H^2}}
\def\H{\mathbb{H}}
\def\G{{\Gamma}}
\def\ve{{\varepsilon}}

\newtheorem{theorem}{Theorem}[section]
\newtheorem{lemma}[theorem]{Lemma}
\newtheorem{proposition}[theorem]{Proposition}

\newtheorem{definition}[theorem]{Definition}

\begin{document}

\begin{title}
{Minimal surfaces with limit ends in $\H^2\times\R$}
\end{title}
\vskip .2in

\begin{author}
{M. Magdalena Rodr\'\i guez\thanks{Research partially supported by a
Spanish MEC-FEDER Grant no. MTM2007-61775, a Regional J. Andaluc\'\i
a Grant no. P09-FQM-5088 and "Grupo Singular" of the UCM.}}
\end{author}
\date{}
\maketitle

\begin{abstract}
  For any $m\geqslant 1$, we construct properly embedded minimal
  surfaces in $\H^2\times\R$ with genus zero, infinitely many vertical
  planar ends and $m$ limit ends. We also provide examples with an
  infinite countable number of limit ends. All these examples are
  vertical bi-graphs.
\end{abstract}


\section{Introduction}
The theory of properly embedded minimal surfaces with genus zero
(i.e. those which are topologically a punctured sphere) in Euclidean
space $\R^3$ has been largely studied
(see~\cite{cmCourant,cm34,mpe3,mpe1} and the references therein).  The
final classification of such surfaces was given by Bill Meeks,
Joaqu\'\i n P\'erez and Antonio Ros in~\cite{mpr6}. The only examples
with infinite topology are Riemann minimal surfaces. They form a
1-parameter family whose natural limits are the catenoid and the
helicoid. When the planar ends are horizontally placed, each Riemann
minimal example is invariant by a non-horizontal translation, its
intersection with any horizontal plane is either a circle or a
straight line, it has infinitely many annular ends asymptotic to
horizontal planes, and has exactly two limit ends\footnote[1]{A limit
  end $e$ of a non-compact surface $M$ is an accumulation point of the
  set $\mathcal{E}(M)$ of ends of $M$. See~\cite{mpe1}.}: one top and one
bottom limit end.

Laurent Hauswirth~\cite{hau1} constructed Riemann-type minimal
surfaces in $\H^2\times\R$, which are properly embedded and have genus
zero, infinitely many ends asymptotic to horizontal slices and two
limit ends: one top and one bottom limit end. It is natural to ask if
there are examples of properly embedded minimal surfaces in
$\H^2\times\R$ with genus zero, infinitely many ends and $m_0$ limit
ends, with $m_0\neq 2$. In this paper we construct examples for any
$m_0\geq 1$, and we also construct examples with an infinite countable
number of limit ends.

In a joint work with Filippo Morabito~\cite{moro1}, we have recently
constructed a $(2k-3)$-parameter family $\mathcal{F}_k$ of properly
embedded minimal surfaces in $\H^2\times\R$ with total (intrinsic)
curvature $4\pi(1-k)$, genus zero and $k$ vertical planar ends (i.e.
annular ends asymptotic to vertical geodesic planes), for any
$k\geqslant 2$. We call them {\it minimal $k$-noids}. Each surface in
this family is invariant by reflection symmetry about the horizontal
slice $\H^2\times\{0\}$. Pyo~\cite{pyo1} has constructed independently
a 1-parameter family of surfaces with the same properties. The
examples given by Pyo, which are included in $\mathcal{F}_k$, are also
invariant by reflection symmetry about $k$ vertical geodesic planes
forming an angle $\pi/k$. The examples with genus zero and infinitely
many ends we construct in the present paper are obtained by taking
limits when $k\to+\infty$ of certain surfaces $M_k\in\mathcal{F}_k$.

The simple ends (i.e. non-limit ends) of the surfaces we construct are
vertical planar ends. Fixed an orientation of $\H^2$, we can order
these simple ends cyclically. If a limit end can be obtained as
accumulation of simple ends ordered following the negative orientation
(resp. the positive orientation) but it cannot be obtained as
accumulation of simple ends ordered following the positive orientation
(resp. negative orientation), we will say that it is a {\it left}
(resp. {\it right}) limit end. In other case, we will say that it is a
{\it 2-sided} limit end.

Now we state the main results of this paper.

\begin{theorem}\label{th:finite}
  For any $m_0\geq 1$, there exists a properly embedded minimal
  surface $\Sigma$ in $\H^2\times\R$ with genus zero, infinitely many
  vertical planar ends and $m_0$ limit ends, which is symmetric with
  respect to a horizontal slice (in fact, it is a vertical bi-graph).
  Moreover, if we denote by $E_\infty^1,\ldots,E_\infty^{m_0}$ the
  limit ends of $\Sigma$, we can prescribe each $E_\infty^m$ to be
  left, right or 2-sided.
\end{theorem}

If we take limits of appropriately chosen minimal surfaces in
$\mathcal{F}_k$, we can also obtain properly embedded minimal surfaces
in $\htwo\times\R$ with genus zero, infinitely many vertical planar
ends and an infinite countable number of limit ends.

\begin{theorem}\label{th:uncountable}
  There exists a properly embedded minimal surface in $\htwo\times\R$
  with genus zero, infinitely many vertical planar ends and an
  infinite countable number of limit ends $\{E_\infty^m\}_{m\in\N}$,
  which is symmetric with respect to a horizontal slice (in fact, it
  is a vertical bi-graph).  Moreover, we can prescribe each
  $E_\infty^m$ to be left, right or 2-sided.
\end{theorem}

All the examples constructed in Theorems~\ref{th:finite}
and~\ref{th:uncountable} are obtained by reflection symmetry about the
horizontal slice $\htwo\times\{0\}$ from a vertical graph contained in
$\htwo\times[0,+\infty)$ whose boundary lies on $\htwo\times\{0\}$.

The author would like to thank Joaqu\'\i n P\'erez for some helpful
conversations.

\section{Preliminaries}
\label{sec:pre}

We consider the half-plane model of $\htwo$,
\[
\H^2=\{(x,y)\in\R^2\ |\ y>0\},
\]
with the hyperbolic metric $g_{-1}=\frac{1}{y^2}(dx^2+dy^2)$.  We
denote by $t$ the coordinate in $\R$ and consider in $\htwo\times\R$
the usual product metric,
\[
ds^2=\frac{1}{y^2}(dx^2+dy^2)+dt^2.
\]

Given an open domain $\Omega \subset {\mathbb H}^2 $ and a smooth
function $u:\Omega \to \R$, the graph of $u$ is a minimal surface in
$\htwo\times\R$ when
\begin{equation}
\label{eq.min.surf}
{\rm div}\left( \frac {\nabla u}
{\sqrt{1+|\nabla u|^2}} \right)=0,
\end{equation}
where all terms are calculated with respect to the metric of $\htwo$.

Finally, we denote by $\partial_\infty\H^2$ the infinite boundary of
$\H^2$, i.e.
\[
\partial_\infty\H^2=\{(x,y)\in\R^2\ |\ y=0\}.
\]

\subsection{Flux of a minimal graph along a curve}
Let $u$ be a minimal graph defined on a domain $\Omega\subset\htwo$.
Assume $\partial \Omega$ is piecewise smooth and $u$ extends
continuously to $\overline{\Omega}$ (possibly with infinite values).
We define the {\it flux} of $u$ along a curve $\Gamma\subset\partial
\Omega$ as
\[
F_u(\Gamma)= \int_\Gamma\left\langle\frac{\nabla u}{\sqrt{1+|\nabla
u|^2}}, \eta\right\rangle ds,
\]
where $\eta$ is the outer normal to $\partial\Omega$ in $\htwo$ and
$ds$ is the arc-length of~$\partial\Omega$.

In the case $\Gamma\subset\O$, we can see $\Gamma$ in the boundary
of different subdomains of $\O$, with two possible induced
orientations. The flux $F_u(\Gamma)$ of $u$ along $\G$ is then
well-defined up to sign, and $|F_u(\Gamma)|$ is well-defined.

Given an arc $C\subset\htwo$, we will denote by $|C|$ the length of
$C$ in $\htwo$. The proof of the following result can be found in
\cite{ner2}, Lemmae 1 and 2.

\begin{lemma}[\cite{ner2}]\label{lem:flux}
  Let $u$ be a minimal graph on a domain $\Omega\subset\htwo$.
  \begin{itemize}
  \item[(i)] For every subdomain $\Omega'\subset\Omega$ such that
    $\overline{\O'}$ is compact, we have $F_u(\partial\Omega')=0$.
  \item[(ii)] Let $C$ be a piecewise smooth curve contained in the
    interior of $\O$, or a convex curve in $\partial\Omega$ where $u$
    extends continuously and takes finite values.  If $C$ has finite
    length, then $|F_u(C)|<|C|$.
  \item[(iii)] Let $T\subset\partial\Omega$ be a geodesic arc of
    finite length such that $u$ diverges to $+\infty$
    (resp. $-\infty$) as one approaches $T$ within $\Omega$. Then
    $F_u(T)=|T|$ (resp. $F_u(T)=-|T|$).
  \end{itemize}
\end{lemma}

The last statement in Lemma~\ref{lem:flux} admits the following
generalization.

\begin{lemma}[\cite{ner2}]\label{lem:limitflux}
  For each $n\in\N$, let $u_n$ be a minimal graph on a fixed domain
  $\Omega\subset\htwo$ which extends continuously to
  $\overline\Omega$, and let $T$ be a geodesic arc of finite length
  in~$\partial\Omega$.
  \begin{itemize}
  \item[(i)] If $\{u_n\}_n$ diverges uniformly to $+\infty$ on compact
    subsets of $T$ while remaining uniformly bounded in compact
    subsets of $\Omega$, then $F_{u_n}(T)\to|T|$.
  \item[(ii)] If $\{u_n\}_n$ diverges uniformly to $+\infty$ in
    compact subsets of $\Omega$ while remaining uniformly bounded on
    compact subsets of $T$, then $F_{u_n}(T)\to-|T|$.
  \end{itemize}
\end{lemma}

\subsection{Divergence lines}
Let $\Omega\subset\htwo$ be a polygonal domain (i.e. a domain whose
edges are geodesic arcs of $\htwo$) with vertices in
$\H^2\cup\partial_\infty\H^2$, possibly infinitely many.  Given a
sequence $\{u_k\}_k$ of minimal graphs defined on $\Omega$, we define
its {\it convergence domain} as
\[
\mathcal{B}=\left\{p\in\Omega\ |\ \{|\nabla u_k(p)|\}_k \mbox{ is
    bounded}\right\},
\]
and the {\it divergence set} of $\{u_k\}_k$ as
\[
\mathcal{D}=\Omega-\mathcal{B}.
\]

From Lemma 4.3 in~\cite{marr1}, we know that the divergence set $\mathcal{D}$ is composed of geodesic arcs contained in~$\O$, called {\it
  divergence lines}, each one joining two points of $\partial\Omega$
(including the vertices of $\Omega$).  The following proposition
describes the convergence domain and the divergence set of a sequence
of minimal graphs. Its proof can be found in~\cite{marr1}, Lemmae 4.2
and 4.3 and Proposition 4.4.

\begin{proposition}[\cite{marr1}]\label{prop:div}
  Let $\Omega\subset\H^2$ be a polygonal domain, and $\{u_k\}_k$ a
  sequence of minimal graphs on $\Omega$. Suppose that $\mathcal{D}$
  is a countable set of divergence lines.  Then passing to a
  subsequence we have:
  \begin{enumerate}
  \item $\mathcal{D}$ is composed of pairwise disjoint geodesic arcs
    contained in~$\O$ (called divergence lines), each one joining two
    points in $\partial\Omega$ (including the vertices of $\Omega$).
  \item $\{|F_{u_k}(T)|\}_k$ converges to $|T|$ as $k\to+\infty$, for
    any geodesic arc $T$ with finite length contained in a divergence
    line $L\subset \mathcal{D}$.
  \item $\mathcal{B}$ is an open set. Moreover, for any component
    $\Omega'$ of $\mathcal{B}$ and any $p\in\Omega'$,
    $\{u_k-u_k(p)\}_k$ converges uniformly on compact subsets of
    $\Omega'$ to a minimal graph $u_\infty$.
  \end{enumerate}
\end{proposition}

\subsection{Jenkins-Serrin graphs on semi-ideal polygonal domains}
\label{subsec:graphs}

Let $\Omega$ be a polygonal domain.  We say that $\Omega$ is {\it
  semi-ideal} when no two consecutive vertices of $\O$ are either in
$\H^2$ or at $\partial_\infty\H^2$ (see Figure~\ref{fig:JScondition}).
We call {\it interior vertices} of $\O$ to those which are contained
in $\htwo$; {\it ideal vertices} of $\Omega$ to those lying on
$\partial_\infty\htwo$; and {\it limit ideal vertices} of $\O$ to the
limit points of ideal vertices of~$\O$.

\begin{figure}
\begin{center}
\includegraphics[scale=.35]{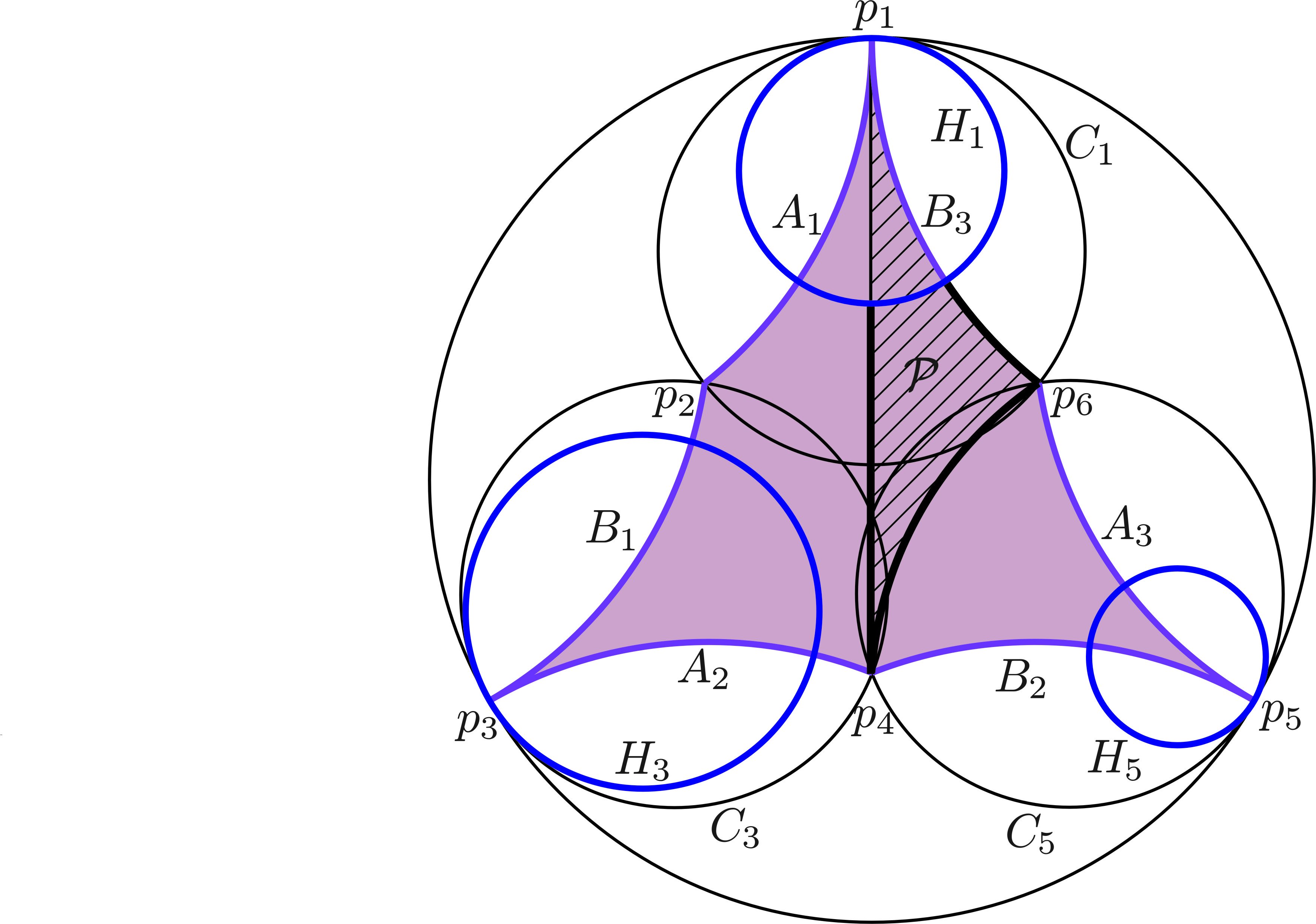}
\caption{Example of a convex Jenkins-Serrin semi-ideal polygonal
domain which verifies condition $(\star)$.} \label{fig:JScondition}
\end{center}
\end{figure}

Fix a semi-ideal polygonal domain $\O$ with a finite number of
vertices $p_1,\ldots,p_{2k}$ (cyclically ordered). We can assume the
odd vertices $p_{2i-1}$ are ideal, and then the even vertices $p_{2i}$
are interior.

For each $i=1,\ldots,k$, we call $A_i$ (resp. $B_i$) the geodesic arc
joining $p_{2i-1},p_{2i}$ (resp.  $p_{2i},p_{2i+1}$).  We consider a
horocycle $H_{2i-1}$ at $p_{2i-1}$.  Assume $H_{2i-1}\cap
H_{2j-1}=\emptyset$ for any $i\neq j$.  Given a polygonal domain
$\mathcal{P}$ inscribed in $\Omega$ (i.e. a polygonal domain
$\mathcal{P}\subset\Omega$ whose vertices are vertices of $\Omega$,
possibly at $\partial_\infty\htwo$), we denote by
$\Gamma(\mathcal{P})$ the part of $\partial \mathcal{P}$ outside the
horocycles (observe that $\Gamma(\mathcal{P})=\partial \mathcal{P}$ in
the case all the vertices of $\mathcal{P}$ are interior). Also let us
call
\[
\a(\mathcal{P})=\sum_{i=1}^k\left|A_i\cap\Gamma(\mathcal{P})\right|
\qquad {\rm and}\qquad
\b(\mathcal{P})=\sum_{i=1}^k\left|B_i\cap\Gamma(\mathcal{P})\right| ,
\]
where we recall that $|\bullet|=\mbox{length}_\htwo(\bullet)$. See
Figure~\ref{fig:JScondition}.

\begin{definition}
  \label{Def:JS}
  {\rm Let $\Omega$ be a semi-ideal polygonal domain with a finite number of
    vertices $p_1,\ldots,p_{2k}$, where
    $p_{2i-1}\in\partial_\infty\htwo$ and $p_{2i}\in\htwo$. We say
    that $\Omega$ is {\it Jenkins-Serrin} if for some choice of
    horocycles $H_{2i-1}$ as above it holds:
    \begin{itemize}
    \item[(i)] $\a(\Omega)=\b(\Omega)$.
    \item[(ii)]
      $2\a(\mathcal{P})<\left|\Gamma(\mathcal{P})\right|$
      and
      $2\b(\mathcal{P})<\left|\Gamma(\mathcal{P})\right|$,
      for every polygonal domain $\mathcal{P}$ inscribed in~$\Omega$,
      $\mathcal{P}\neq\Omega$.
    \end{itemize}
  }
\end{definition}

We remark that condition (i) in the above definition does not depend
on the choice of horocycles; and if the inequalities of condition (ii)
are satisfied for some choice of horocycles, then they continue to
hold for ``smaller'' horocycles (see the argument given by Pascal
Collin and Harold Rosenberg in~\cite{cor2}, pages 1884 and 1885). The
following result is a particular case of Theorem 4.12 in~\cite{marr1}.

\begin{theorem}[\cite{marr1}]
  \label{th:JS} Let $\Omega$ be a
  semi-ideal polygonal domain with edges $A_1, B_1,\ldots,A
  _k, B_k$ (cyclically
  ordered).  There exists a solution $u$ for the minimal graph
  equation~\eqref{eq.min.surf} in $\O$ with boundary values
  \[
  u|_{A_i}=+\infty\quad\mbox{and}\quad u|_{B_i}=-\infty, \quad
  \mbox{for any }\ i=1,\ldots, k,
  \]
  if, and only if, $\Omega$ is a Jenkins-Serrin domain. Moreover, such
  a solution is unique up to an additive constant, when it exists.
\end{theorem}

\begin{figure}
\begin{center}
\includegraphics[scale=.35]{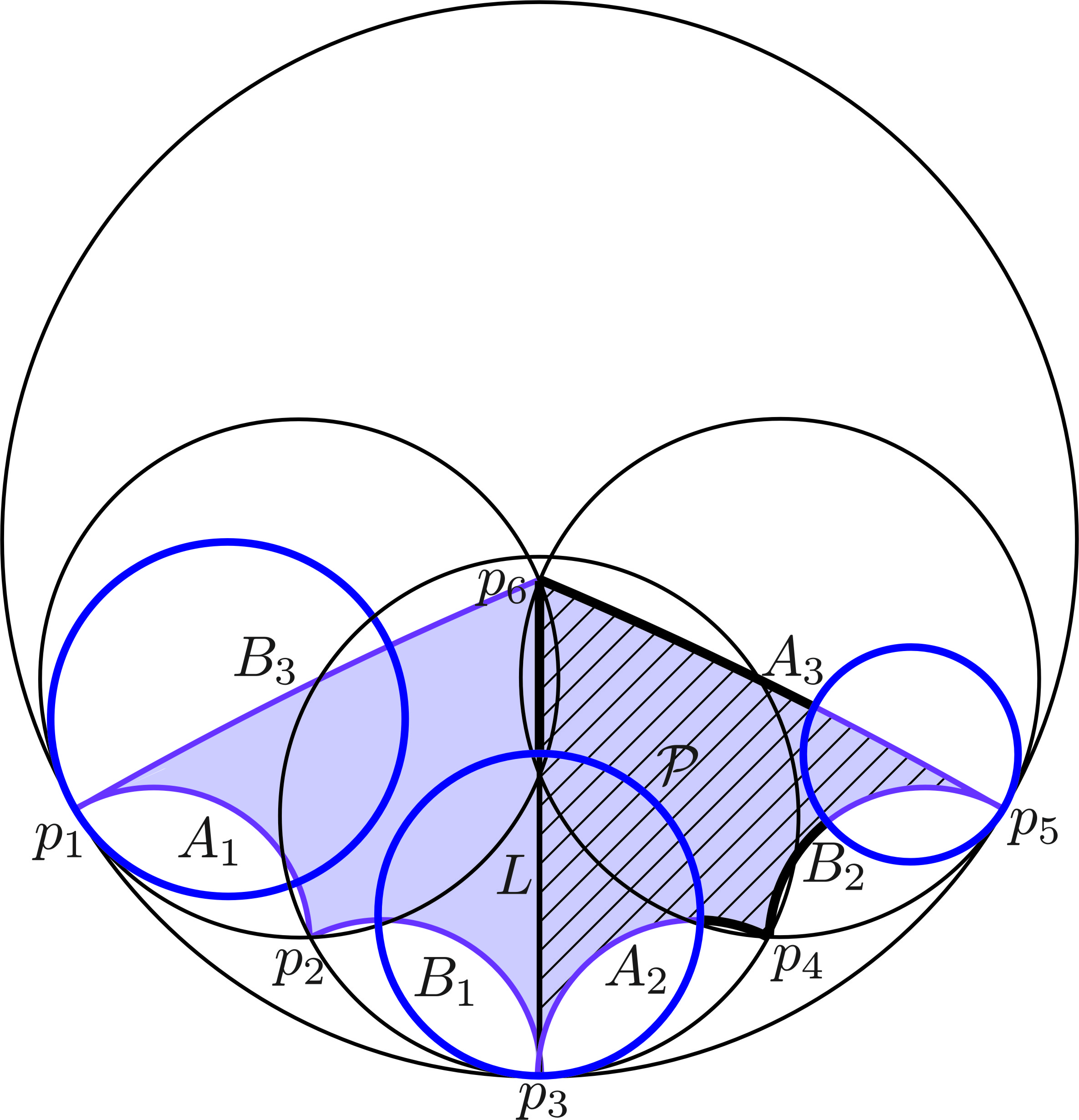}
\caption{Example of a convex semi-ideal polygonal domain satisfying
  condition $(\star)$ which is not Jenkins-Serrin, since
  $2\a(\mathcal{P})>|\Gamma(\mathcal{P})|$ because
  $|A_2\cap\Gamma(\mathcal{P})|>|L\cap\Gamma(\mathcal{P})|$.}
\label{fig:nonJS}
\end{center}
\end{figure}

We will work with convex semi-ideal polygonal domains $\O$
satisfying the following additional condition (see
Figures~\ref{fig:JScondition} and~\ref{fig:nonJS}):

\medskip

\begin{tabular}{cll}
  &  ($\star$) &  {\it There
    exists a choice of pairwise disjoint horocycles $H_{2i-1}$ at the ideal}\\
  & & {\it  vertices $p_{2i-1}\in
    \partial_\infty\htwo$ such that}\\[2mm]
  & &
  \mbox{}\hspace{2cm}$\mbox{dist}_\htwo(p_{2i-2},H_{2i-1})=
  \mbox{dist}_\htwo(p_{2i},H_{2i-1})$\\[2mm]
  & & for any $i=1,\ldots,k$, using the cyclical notation $p_0=p_{2k}$.
\end{tabular}

We remark that condition ($\star$) does not depend on the choice of
horocycles $H_{2i-1}$, and it is equivalent to the existence of a
horocycle $C_{2i-1}$ at $p_{2i-1}$ passing through
$p_{2i-2},p_{2i}$, for any $i=1,\ldots,k$. We call $D_{2i-1}$ the
component of $\htwo-C_{2i-1}$ whose only point of
$\partial_\infty\htwo$ at its infinite boundary is~$p_{2i-1}$ (i.e.
$D_{2i-1}$ is the horodisk at $p_{2i-1}$ bounded by $C_{2i-1}$), and
$\overline{D_{2i-1}}=D_{2i-1}\cup C_{2i-1}$.

Before finishing this subsection, we describe geometrically when a
semi-ideal polygonal domain with a finite number of vertices and
satisfying condition ($\star$) is a Jenkins-Serrin domain. See
Figure~\ref{fig:nonJS}.

\begin{lemma}\label{lem:car}
  Let $\O$ be a semi-ideal polygonal domain with vertices $p_1,\ldots,p_{2k}$
  cyclically ordered so that $p_{2i-1}\in\partial_\infty{\mathbb H}^2$
  and $p_{2i}\in\htwo$, for any $i=1,\ldots,k$. Suppose $\O$ satisfies
  condition ($\star$) above.  Then the following assertions are
  equivalent:
  \begin{enumerate}
  \item $\O$ is a Jenkins-Serrin domain.
  \item $p_{2j}\in\H^2-\overline{D_{2i-1}}$, for any $j$ and any
    $i\not\in\{j,j+1\}$.
  \end{enumerate}
\end{lemma}
\begin{proof}
  Before proving the lemma, let us fix some notation.  For any
  $i=1,\ldots,k$, consider the nested sequence of horocycles
  $\{H_{2i-1}(n)\}_n$ at $p_{2i-1}$ contained in $D_{2i-1}$ and
  converging to $p_{2i-1}$ as
  $n\to+\infty$, such that
  $\mbox{dist}_\htwo\left(H_{2i-1}(n), C_{2i-1}\right) =n$,
  for any~$n$.  Then
  \[
  \mbox{dist}_\htwo(p_{2i-2},H_{2i-1}(n))
  =\mbox{dist}_\htwo(p_{2i},H_{2i}(n))=n.
  \]

  Let us now prove Lemma~\ref{lem:car}.
  First suppose $\O$ is Jenkins-Serrin and there exists some
  $p_{2j}\in \overline{D_{2i-1}}$, with $i\not\in\{j,j+1\}$.  We then
  have
  \[
  \mbox{dist}_\htwo(p_{2j},H_{2i-1}(n))\leqslant n
  \]
  for $n$ large.  Let $L$ be the geodesic arc from $p_{2j}$ to
  $p_{2i-1}$, and $\mathcal{P}$ be the component of $\O-$
  containing $A_i$ on its boundary (see Figure~\ref{fig:nonJS}).
  Clearly, $\mathcal{P}$ is a polygonal domain inscribed in $\O$. And,
  for this choice of horocycles $H_{2i-1}(n)$, it holds
  $\left|A_\ell\cap\Gamma(\mathcal{P})\right|=n$ (resp.
  $\left|B_\ell\cap\Gamma(\mathcal{P})\right|=n$) for any $\ell$ such
  that $A_\ell\subset\partial{\mathcal P}$
  (resp. $B_\ell\subset\partial{\mathcal P}$). Thus
  $\beta(\mathcal{P})=\a(\mathcal{P})-n$, and then
  \[
  |\Gamma(\mathcal{P})|=\mbox{dist}_\htwo(p_{2j},H_{2i-1}(n))
  +\a(\mathcal{P})+\beta(\mathcal{P})\leqslant 2\a(\mathcal{P}).
  \]
  This holds for every $n$ large, which contradicts that $\O$ is a
  Jenkins-Serrin domain. This proves
  $\emph{(1)}\Rightarrow\emph{(2)}$.

  Now assume $p_{2j}\in\H^2-\overline{D_{2i-1}}$, for any $j$ and any
  $i\not\in\{j,j+1\}$, and let us prove that $\O$ is a Jenkins-Serrin
  domain.  As we have remarked above, we have $\a(\O)=\b(\O)$.
  Suppose there exists an inscribed polygonal domain $\mathcal{P}$ in
  $\O$, $\mathcal{P}\neq\O$, such that
  \[
  |\Gamma(\mathcal{P})| \leqslant 2\a(\mathcal{P})
  \]
  (the case $|\Gamma(\mathcal{P})| \leqslant 2\b(\mathcal{P})$ follows
  similarly).  Since $\mathcal{P}\neq\O$, there is at least an
  interior geodesic $\g_1$ in $\partial{\mathcal{P}}$ (i.e.
  $\g_1\subset\partial{\mathcal{P}}\cap\O$). We can assume there are
  no two consecutive interior geodesics $\g_1,\g_2$ in
  $\partial{\mathcal{P}}$: We would replace $\mathcal{P}$ by another
  inscribed polygonal domain satisfying the same properties as
  $\mathcal{P}$ by replacing $\g_1\cup\g_2$ by the geodesic $\g_3$
  such that $\g_1\cup\g_2\cup\g_3$ bounds a geodesic triangle
  contained in $\O$.  In a similar way, we can assume that
  \[
  \partial{\mathcal{P}}= A_{i_1}\cup\g_1\cup\ldots\cup
  A_{i_j}\cup\g_j\cup A_{i_{j+1}}\cup\ldots\cup A_{i_s}\cup\g_s ,
  \]
  where each $\g_j$ is either an interior geodesic or a $B_i$ edge,
  and at least $\g_1\subset\O$.  In particular, each $\g_j$ joins an
  even vertex $p_{2i_j}$ to an odd vertex $p_{2i_{j+1}-1}$.  (Observe
  that, when $\g_j$ is a $B_i$ edge, then $\g_j=B_{i_j}$ and
  $i_{j+1}=i_j+1$.)

  Hence $\sum_{j=1}^s |\g_j\cap\Gamma(\mathcal{P})|=
  |\Gamma(\mathcal{P})|-\a(\mathcal{P})\leqslant \a(\mathcal{P})=s n$,
  from where we deduce there must be an interior geodesic
  $\g_j\subset\partial{\mathcal{P}}$ whose length is smaller than or
  equal to $n$. But this implies the vertex $p_{2i_j}$ lies on
  $\overline D_{2 i_{j+1}-1}$ and $i_j\not\in\{i_{j+1}-1,i_{j+1}\}$, a
  contradiction.
\end{proof}

\subsection{Conjugate surfaces in $\htwo\times\R$}
\label{subsec:conjugation} In this subsection we briefly recall how to
obtain minimal surfaces in $\htwo\times\R$ by conjugation from other
known minimal examples.  For more details see Daniel~\cite[Section
4]{da2} and Hauswirth, Sa Earp and Toubiana~\cite{HST}.

Let $\Sigma$ be a 2-sided minimal surface in $\htwo\times\R$. We call
height function of $\Sigma$ to the horizontal projection
$h:\Sigma\to\R$, which is known to be a real harmonic map. And we
denote by $F:\Sigma\to\htwo$ the vertical projection, which is a
harmonic map, and by
\[
Q=\langle F_z \bar{F}_z\rangle (dz)^2
\]
the Hopf differential associated to $F$, where $z$ is a local
conformal coordinate on~$\Sigma$.  Finally, we denote by $N$ a
globally defined unit normal vector field on $\Sigma$ and by
$\nu=\langle N,\frac{\partial}{\partial t} \rangle$ the angle function
of $\Sigma$.

\begin{theorem}[\cite{da2,HST}]\label{th:conjugate}
  Let $\Sigma$ be a simply-connected minimal surface in
  $\htwo\times\R$.  There exists a minimal surface
  $\Sigma^*\subset\htwo\times\R$, called {\rm conjugate surface of
    $\Sigma$}, such that:
\begin{enumerate}
\item $\Sigma$ and $\Sigma^*$ are isometric.  (If we identify points
  in $\Sigma$ and $\Sigma^*$ via an isometry, we can assume that the
  angle function $\nu^*$, the height function $h^*$, the vertical
  projection $F^*$ of $\Sigma^*$, and the Hopf differential $Q^*$
  associated to $F^*$, are all defined on $\Sigma$.)
\item The angle functions $\nu,\nu^*$ coincide.
\item The height functions $h,h^*$ are real harmonic conjugate.
\item $Q^*=-Q$.
\end{enumerate}
The conjugate surface $\Sigma^*$ is well-defined up to an isometry
of $\H^2\times\R$. Finally, the conjugation exchanges the following
Schwarz reflections:
  \begin{itemize}
  \item The symmetry with respect to a vertical geodesic plane of
    $\H^2\times\R$ containing a geodesic curvature line of $\Sigma$ becomes the
    rotation of $\H^2\times\R$ by angle $\pi$ with respect to a horizontal geodesic
    contained in $\Sigma^*$, and viceversa.
  \item The symmetry with respect to a horizontal slice
    containing a geodesic curvature line of $\Sigma$ becomes the rotation  by angle $\pi$  with
    respect to a vertical straight line contained in $\Sigma^*$, and
    viceversa.
\end{itemize}
\end{theorem}

We will use the above correspondence to study the conjugate surface
of a minimal graph defined on a convex semi-ideal polygonal domain
of $\htwo$.  The surface constructed in this way is a minimal graph
(and consequently embedded), as ensured by the following Krust-type
theorem.

\begin{theorem}[\cite{HST}]\label{th:krust}
  If $\Sigma$ is a minimal graph over a convex domain $\O$ of $\htwo$,
  then $\Sigma^*$ is also a minimal graph over a (non-necessarily
  convex) domain $\O^*\subset\htwo$.
\end{theorem}

\subsection{Minimal $k$-noids of $\htwo\times\R$}
\label{subsec:semi-ideal} In this subsection we briefly explain the
construction of the properly embedded minimal surfaces of
$\H^2\times\R$ given in~\cite{moro1,pyo1}, which have genus zero,
$k\geq 2$ vertical planar ends and finite total (intrinsic) curvature
$4\pi(1-k)$. We call them minimal $k$-noids of $\htwo\times\R$.

Let $\O$ be a convex Jenkins-Serrin semi-ideal polygonal domain with
$2k$ vertices $p_1, \ldots, p_{2k}$, cyclically ordered, so that the
even vertices $p_{2i}$ are located in the interior of $\htwo$, and
the odd vertices $p_{2i-1}$ are at $\partial_\infty{\mathbb H}^2$,
for $i=1,\ldots,k$. We call $A_i$ the edge of $\O$ whose endpoints
are $p_{2i-1},p_{2i}$, and $B_i$ the edge of $\O$ whose endpoints
are $p_{2i},p_{2i+1}$.  We also require that $\O$ satisfies the
condition~($\star$) defined in Subsection~\ref{subsec:graphs}.

By Theorem~\ref{th:JS}, there exists a unique solution $u$ to the
minimal graph equation~\eqref{eq.min.surf} defined over $\O$ with
boundary values $+\infty$ on $A_i$ and $-\infty$ on $B_i$ such that
$u(p_0)=0$, for some fixed point $p_0\in\Omega$.  Denote by $\Sigma$
the graph surface of $u$; $\Sigma$ is bounded by the $k$ vertical
straight lines $\Gamma_i=\{p_{2i}\}\times\R$, $i=1,\ldots,k$.

\begin{figure}
\begin{center}
\includegraphics[scale=.3]{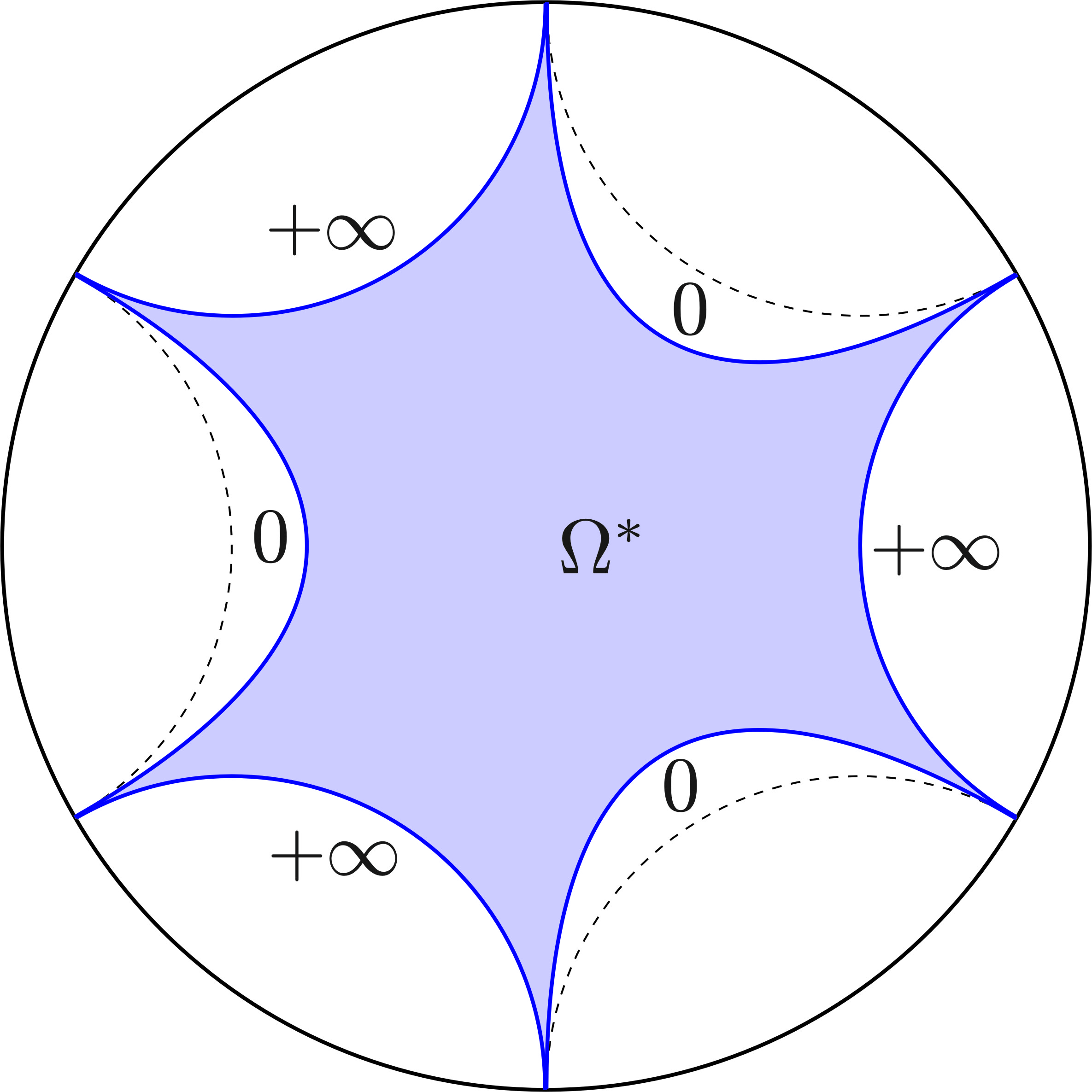}
\caption{Vertical projection of the conjugate surface $\Sigma^*$ in
a symmetric case.}
\label{fig:conjugate}
\end{center}
\end{figure}

The conjugate surface $\Sigma^*$ of $\Sigma$ is a minimal graph over
a (non-necessarily convex) domain $\O^*\subset\htwo$, by
Theorem~\ref{th:krust} (see Figure~\ref{fig:conjugate}). And
$\partial\Sigma^*$ consists of $k$ horizontal geodesic curvature
lines $\Gamma_i^*$. In~\cite{moro1} it is proved that
$\Gamma_i^*\subset\htwo\times\{0\}$ for any $i$ and that $\Sigma^*$
is contained in one of the half-spaces determined by
$\htwo\times\{0\}$. By reflecting $\Sigma^*$ with respect to
$\htwo\times\{0\}$, we get a properly embedded minimal surface $M$
with genus zero and $k$ vertical planar ends, which has total
(intrinsic) curvature $4\pi(1-k)$. The ends of $M$ are asymptotic to
the vertical geodesic planes $\eta_i^*\times\R$, where the
$\eta_i^*$ are the complete geodesics such that
$\partial\O^*=\G_1^*\cup\eta_1^*\cup\ldots\cup \G_k^*\cup\eta_k^*$
(cyclically ordered).

\section{Proof of Theorem~\ref{th:finite}: examples with $m_0$ limit
  ends}
\label{sec:finite}

Firstly, let us recall some definitions. A limit end $e$ of a
non-compact surface $M$ is an accumulation point of the set
$\mathcal{E}(M)$ of ends of $M$. This makes sense since
$\mathcal{E}(M)$ can be endowed with a natural topology for which it
is a compact, totally disconnected subspace of the real interval
$[0,1]$. See~\cite{mpe1} for more details. We call simple ends of $M$
to its non-limit ends.

Assume the simple ends of $M$ are asymptotic to vertical geodesic
planes (called vertical planar ends) which can be ordered
cyclically\footnote[2]{The surface $M$ we want to construct will be
  obtained as a limit of minimal $k$-noids, and it is got by
  reflection symmetry from a minimal graph with boundary values
  $0,+\infty$, alternately. The vertical projection of $M$ will be
  bounded by strictly concave curves $\Gamma_i^*$ and geodesic curves
  $\eta_i^*$, disposed alternately and asymptotic at
  $\partial_\infty\htwo$. The ends of $M$ will be asymptotic to the
  vertical geodesic planes $\eta_i^*\times\R$, which are cyclically
  ordered.}, fixed an orientation of $\partial_\infty\htwo$. If the
limit end can be obtained as accumulation of simple ends ordered
following the negative orientation (resp. the positive orientation)
but it cannot be obtained as accumulation of simple ends ordered
following the positive orientation (resp.  negative orientation), we
will say that it is a {\it left} (resp.  {\it right}) limit end. In
other case, we will say that it is a {\it 2-sided} limit end.

In this section we construct properly embedded minimal surfaces in
$\htwo\times\R$ with genus zero, infinitely many vertical planar ends
and $m_0$ limit ends, for any $m_0\geq 1$. Furthermore, we can
prescribe the behavior of the limit ends; more precisely, if we denote
by $E_\infty^1,\ldots,E_\infty^{m_0}$ (cyclically ordered) the limit
ends of such a surface, we can prescribe if each $E_\infty^m$ is
either a left, a right or a 2-sided limit end.

Now let us explain our construction. In a first step we will
construct, by taking limits of convex Jenkins-Serrin semi-ideal
polygonal domains $\O_k$ with finitely many vertices, a convex
semi-ideal polygonal domain $\O_\infty$ with an infinite countable set
$\mathcal{S}$ of ideal vertices and ${m_0}$ limit points
$p_\infty^1,\ldots,p_\infty^{m_0}$ (cyclically ordered) such that, if
$E_\infty^m$ is prescribed to be a left (resp. a right or a 2-sided)
limit end, then $p_\infty^m$ is a left (resp. a right or a 2-sided)
limit point, see Definition~\ref{def:limit} below. We will say that
such a limit point $p_\infty^m$ is a {\it left} (resp. {\it a right or
  a 2-sided}) {\it limit ideal vertex} of $\O_\infty$.

\begin{definition}\label{def:limit}
  {\rm Let $\mathcal{S}$ be a set of points in
    $\partial_\infty\htwo=\{y=0\}$. We will say that
    $p_\infty\in\mathcal{S}$ is a {\rm limit point} if every
    neighborhood of $p_\infty$ in $\partial_\infty\htwo$ contains a
    point of $\mathcal{S}$ other than $p_\infty$ itself; i.e. if
    $p_\infty=(x_\infty,0)$, then $\mathcal{S}\cap\{y=0,\
    0<|x-x_\infty|<\ve\}\neq\emptyset$ for every $\ve>0$; or
    $p_\infty=\infty$ and $\mathcal{S}\cap\{y=0,\
    |x|>M\}\neq\emptyset$ for every $M>0$.

    A limit point $p_\infty=(x_\infty,0)$ of $\mathcal{S}$ is said to
    be a {\rm left} (resp. a {\rm right}) limit point if there exists
    some $\ve>0$ such that $\mathcal{S}\cap
    \{-\ve<x-x_\infty<0\}=\emptyset$ (resp. $\mathcal{S}\cap
    \{0<x-x_\infty<\ve\}=\emptyset$); and it is said to be a {\rm
      2-sided} limit point in other case.

    If $p_\infty=\infty$ is a limit point of $\mathcal{S}$, we say
    that it is a {\rm left} (resp. a {\rm right}) limit point if there
    exists some $M>0$ such that $\mathcal{S}\cap \{x>M\}=\emptyset$
    (resp.  $\mathcal{S}\cap \{x<-M\}=\emptyset$); and it is a {\rm
      2-sided} limit point in other case. }
\end{definition}

Next we will get a Jenkins-Serrin minimal graph $\Sigma$ over
$\O_\infty$ as a limit of Jenkins-Serrin minimal graphs over the
$\O_k$ domains. Finally, we will prove that the conjugate surface of
$\Sigma$ is a minimal graph $\Sigma^*\subset\htwo\times[0,+\infty)$
whose boundary, which consists of horizontal geodesic curvature lines,
is contained in the horizontal slice $\htwo\times\{0\}$. The desired
surface is obtained from $\Sigma^*$ by reflection symmetry about
$\htwo\times\{0\}$.

\subsection{Construction of the domains}
\label{subsec:finite}
This subsection deals with the construction of the convex semi-ideal
polygonal domain $\O_\infty$ in the argument explained above. We will
construct a sequence of convex Jenkins-Serrin semi-ideal polygonal
domains $\O_k$ satisfying condition $(\star)$ defined in
Subsection~\ref{subsec:graphs}, each $\O_k$ with a finite number of
vertices, with $\O_k\subset\O_{k+1}$ for any $k$, and such that they
converge to a domain $\O_\infty$ in the desired conditions.

Consider $m_0$ different ideal points in $\partial_\infty\htwo$:
\[
p^1_\infty=\infty,\ p^2_\infty=(x_\infty^2,0),\ldots,\
p^{m_0}_\infty=(x_\infty^{m_0},0) ,
\]
with $-\infty<x_\infty^2<\ldots<x_\infty^{m_0}<+\infty$, when $m_0\geq
2$; in the case $m_0=1$, we only have $p_\infty^1=\infty$.  These
points will be the limit ideal vertices of $\O_\infty$.

We call $\mathcal{M}=\{ m\in\N\ |\ 1\leq m\leq m_0\}$. For any
$m\in\mathcal{M}$, choose two ideal points $p_{-1}^m=(x_{-1}^m,0),\
p_{1}^m=(x_{1}^m,0)$ with
$x_\infty^{m}<x_{-1}^m<x_{1}^m<x_\infty^{m+1}$, where
$x_\infty^{1}=-\infty$ and $x_\infty^{m_0+1}=+\infty$.

\begin{figure}
  \begin{center}
    \includegraphics[scale=.4]{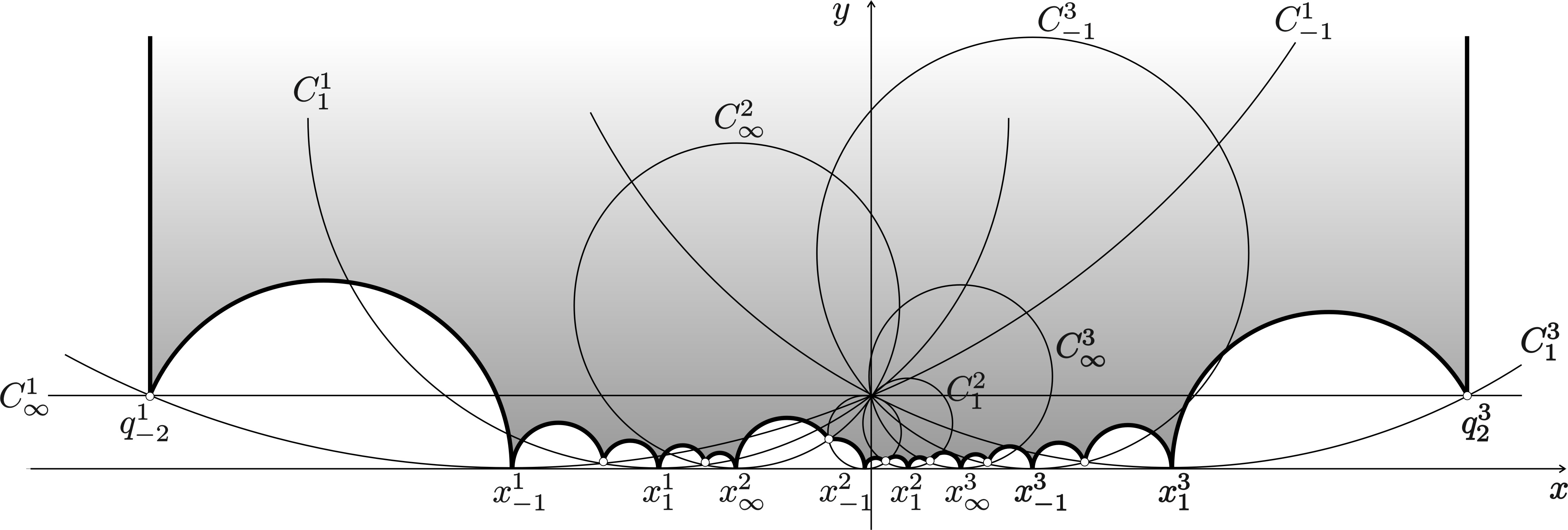}
    \caption{The shadowed region is an example of $\O_1$ for three
      limit ideal vertices, where ${p^1_\infty=\infty}$ and
      $p^m_i=(x^m_i,0)$ for the remaining values of $i,m$. The
      interior vertices in white correspond, from the left to the
      right, to
      $q_{-2}^1,p_0^1,q_2^1,q_{-2}^2,p_0^2,q_2^2,q_{-2}^3,p_0^3,q_2^3$.}
    \label{fig:Omega1_centrado}
  \end{center}
\end{figure}

For any $m\in\mathcal{M}$ and any $j\in\{-1,1,\infty\}$, we call
$C^m_j$ the horocycle at $p^m_j$ passing through $P_0=(0,1)$.  Denote
by $p_0^m$ (resp. $q_{-2}^m,\ q_2^m$) the point in $C_{-1}^m\cap
C_1^m$ (resp. $C_\infty^{m}\cap C_{-1}^m$, $C_1^m\cap C_\infty^{m+1}$)
which is different from $P_0$, see
Figure~\ref{fig:Omega1_centrado}. We define $\O_1$ as the semi-ideal
polygonal domain with set of vertices
\[
\{p_\infty^m, q_{-2}^m, p_{-1}^m, p_0^m, p_1^m, q_2^m\ |\
m\in\mathcal{M}\}.
\]
By definition of $q_{-2}^m,p_0^m,q_2^m$, it is clear that $\O_1$
satisfies condition~$(\star)$. Now let us see we can choose the ideal
vertices $p_{-1}^m, p_{1}^m$ to assure $\O_1$ is convex. It suffices
to choose appropriately $x_{-1}^m,x_{1}^m$ such that
$q_{-2}^m,p_0^m,q_2^m\subset\{0<y<1\}$, except for $q_{-2}^1,
q_2^{m_0}\in \{y=1\}$.
\begin{itemize}
\item If $x_\infty^{m}\geq 0$ or $x_\infty^{m+1}\leq 0$, then
  $x_{-1}^m,x_{1}^m$ can be chosen arbitrarily.
\item In the case $x_\infty^{m}< 0<x_\infty^{m+1}$, we take
  $\max\{x_\infty^{m},-1\}< x_{-1}^m<0< x_{1}^m<
  \min\{x_\infty^{m+1},1\}$.
\end{itemize}
With the choice above, the domain $\O_1$ is convex. Finally, let us
check that $\O_1$ is a Jenkins-Serrin domain. Using
Lemma~\ref{lem:car} it suffices to get that $p_{0}^m$ (resp.
$q_{-2}^m$; $q_{2}^m$) lies outside $C_{j}^{m'}$, for any
$j\in\{-1,1,\infty\}$ and any $m'\in\mathcal{M}$ such that $C_{j}^{m'}$
is different from $C_{-1}^m$, $C_{1}^m$ (resp. $C_{\infty}^m$,
$C_{-1}^m$; $C_{1}^m$, $C_{\infty}^{m+1}$). By the choice above, this
is the case when $C_{j}^{m'}=C_{\infty}^{1}$. Let us assume
$C_{j}^{m'}\neq C_{\infty}^{1}$. We prove it for $p_{0}^m$ (for
$q_{-2}^m$ $q_{2}^m$ it can be obtained similarly): If we denote
$p_0^m=(x_0^m,y_0^m)$, we have $x_{-1}^m<x_0^m<x_{1}^m$. If
$x_{j}^{m'}>x_{1}^m$ (resp. $x_{j}^{m'}<x_{-1}^m$), then $C_{j}^{m'}$
divides $C_{1}^m$ (resp. $C_{-1}^m$) in two components, one of them
containing both $p_{0}^m$ and $p_{1}^m$ (resp.  $p_{-1}^m$). That says
that $p_{0}^m$ is outside $C_{j}^{m'}$.  \medskip

Now we consider the subsets of $\mathcal{M}$ given by
\[
\mathcal{M}^+=\{ m\in\mathcal{M}\ |\ E_\infty^{m+1}\ \mbox{ is
prescribed to be either a right or a 2-sided limit end}\},
\]
\[
\mathcal{M}^-=\{ m\in\mathcal{M}\ |\ E_\infty^{m}\ \mbox{ is
prescribed to be either a left or a 2-sided limit end}\}.
\]
For any $k\geq 2$, we define $\O_k$ as the semi-ideal polygonal
domain with set of vertices $\mathcal{V}_k^-\cup \mathcal{V}^0\cup
\mathcal{V}_k^+$, where
\[
\mathcal{V}_k^-=\{q_{-2k}^m,p_{1-2k}^m,p_{2-2k}^m,\ldots,
p_{-3}^m,p_{-2}^m\ |\ m\in\mathcal{M}^-\} \cup \{q_{-2}^m\ |\
m\in\mathcal{M}-\mathcal{M}^-\},
\]
\[
\mathcal{V}_k^0= \{p_\infty^m,p_{-1}^m,p_0^m,p_1^m\ |\
m\in\mathcal{M}\}  ,
\]
\[
\mathcal{V}_k^+=\{p_{2}^m,p_{3}^m,\ldots,
p_{2k-2}^m,p_{2k-1}^m,q_{2k}^m\ |\ m\in\mathcal{M}^+\} \cup
\{q_{2}^m\ |\ m\in\mathcal{M}-\mathcal{M}^+\}.
\]
and the vertices $p_{\pm i}^m,q_{\pm 2k}^m$ are defined by induction
as follows:

\begin{enumerate}
\item Suppose that $m\in\mathcal{M}^+$ and that we have defined the
  ideal vertices
  \[
  p_1^m=(x_1^m,0),\ \ldots,\ p_{2k-1}^m=(x_{2k-1}^m,0),
  \]
  with $k\geq 1$ and $x_{1}^m<\ldots<x_{2k-1}^m<x_\infty^m$. These
  ideal vertices determine the following data: For $1\leq i\leq k$,
  \begin{itemize}
  \item let $C_{2i-1}^m$ be the horocycle at $p_{2i-1}^m$ passing
    through $P_0$;
  \item $p_{2i-2}^m$ is defined as the intersection point in
    $C_{2i-3}^m\cap C_{2i-1}^m$ different from $P_0$;
  \item $q_{2i}^m$ is the intersection point in $C_{2i-1}^m\cap
    C_\infty^{m+1}$ different from $P_0$.
  \end{itemize}
  This choice of $p_{2i-2}^m,q_{2i}^m$ will assure that $\O_k$ is a
  convex Jenkins-Serrin semi-ideal polygonal domain which satisfies
  condition~$(\star)$.

  Let us now define $p_{2k+1}^m$. We call $\Gamma_{2k}^m$ (resp.
  $\gamma_{2k}^m$) the complete geodesic curve with
  endpoint~$p_{2k-1}^m$ (resp. $p^{m+1}_\infty$) passing through
  $q_{2k}^m$. Let $(a_{2k}^m,0)$ (resp. $(b_{2k}^m,0)$) be the
  endpoint of $\Gamma_{2k}^m$ (resp. $\gamma_{2k}^m$) different
  from~$p_{2k-1}^m$ (resp. $p^{m+1}_\infty$). We take
  $p_{2k+1}^m=(x_{2k+1}^m,0)$ satisfying $b_{2k}^m\leq x_{2k+1}^m\leq
  a_{2k}^m$. We consider that property for $p_{2k+1}^m$ in order to
  get $\O_k\subset\O_{k+1}$.

  We remark that both $p_{2k+1}^m$ and $p^m_{2k}$ converge to
  $p_\infty^{m+1}$ as $k\to +\infty$.

  \begin{figure}
    \begin{center}
      \includegraphics[scale=0.8]{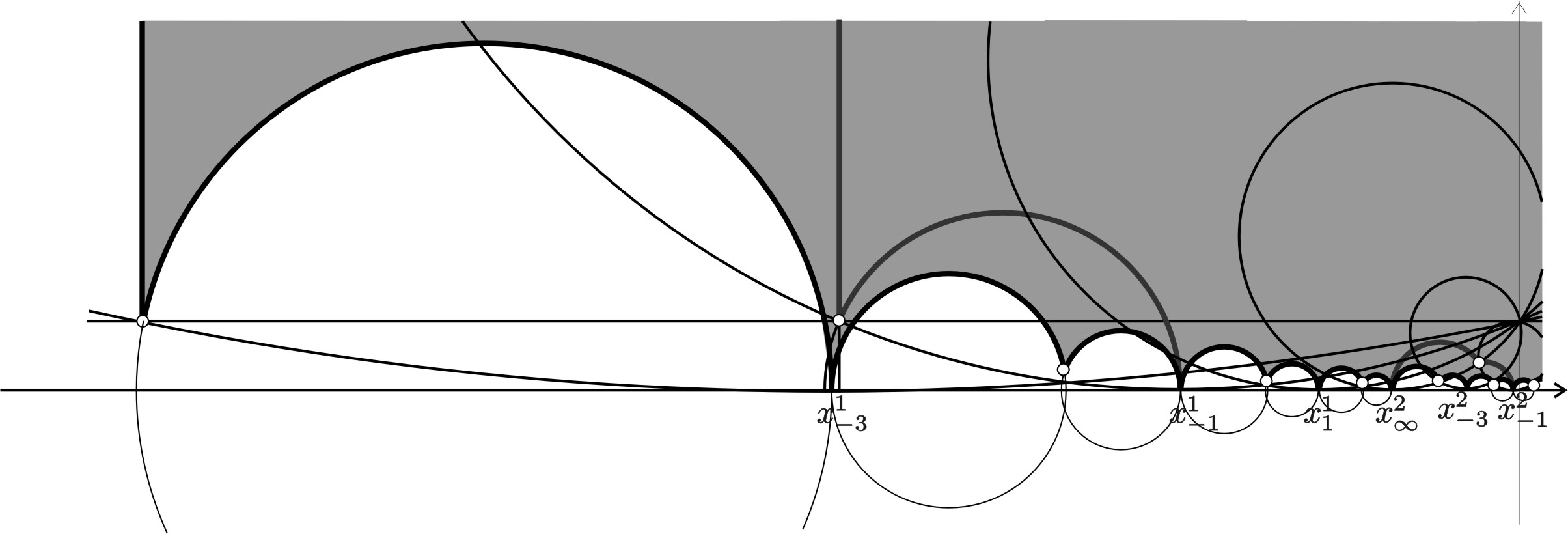}
      \caption{The shadowed region is a piece of $\O_2$, with
        $1\in\mathcal{M}^- -\mathcal{M}^+$ and
        $2\in\mathcal{M}^-$.} \label{fig:finito}
    \end{center}
  \end{figure}

\item The corresponding definition for $m\in\mathcal{M}^-$ follows
  analogously: Suppose $m\in\mathcal{M}^-$ and that, for $k\geq 1$, we
  have defined the ideal vertices
  \[
  p_{1-2k}^m=(x_{1-2k}^m,0),\ \ldots,\ p_{-1}^m=(x_{-1}^m,0) ,
  \]
  with $x_\infty^{m}<x_{1-2k}^m<\ldots<x_{-1}^{m}$. These ideal
  vertices determine the following data: For $1\leq i\leq k$,
  \begin{itemize}
  \item let $C_{1-2i}^m$ be the horocycle at $p_{1-2i}^m$ passing
    through $P_0$;
  \item $p_{2-2i}^m$ is defined as the intersection point in
    $C_{1-2i}^m\cap C_{3-2i}^m$ different from $P_0$;
  \item $q_{-2i}^m$ is the intersection point in $C_\infty^{m}\cap
    C_{1-2i}^m$ different from $P_0$.
  \end{itemize}

  Let us now define $p_{-1-2k}^m$.  We call $\Gamma_{-2k}^m$ (resp.
  $\gamma_{-2k}^m$) the complete geodesic curve with
  endpoint~$p_{1-2k}^m$ (resp. $p^{m}_\infty$) passing through
  $q_{-2k}^m$. Let $(a_{-2k}^m,0)$ (resp. $(b_{-2k}^m,0)$) be the
  endpoint of $\Gamma_{-2k}^m$ (resp. $\gamma_{-2k}^m$) different
  from~$p_{1-2k}^m$ (resp. $p^{m}_\infty$). We choose
  $p_{-1-2k}^m=(x_{-1-2k}^m,0)$, with $b_{-2k}^m\leq x_{-1-2k}^m\leq
  a_{-2k}^m$. With this choice of ideal vertices, we get
  $\Omega_k\subset\Omega_{k+1}$ and that both $p_{-1-2k}^m$ and
  $p_{-2k}$ converges to $p_\infty^{m}$ as $k\to +\infty$.
\end{enumerate}

By definition of the interior vertices $p_{2i-2}^m,q_{2i}^m$, the
semi-ideal polygonal domain $\O_k$ satisfies condition~$(\star)$. As
$x_{2k-1}^m$ has the same sign as $x_{2k-3}^m$ and $x_\infty^{m+1}$,
then $p_{2i-2}^m,q_{2i}^m\subset\{0<y<1\}$. That fact assures that
$\O_k$ is convex. Moreover, as the horocycles $C_j^{m'}$ can be
ordered from the left to the right and they all pass through $P_0$, we
can deduce (as in the case of $\O_1$) that the interior vertices
$p_{2i}^m$, $q_{-2k}^m$, $q_{2k}^m$ are outside the horocycles
$C_j^{m'}$, except for those used for defining them (i.e. their
consecutive ones). Then $\O_k$ is a Jenkins-Serrin domain, by
Lemma~\ref{lem:car}.

Finally, we have defined the ideal vertices $p_{2i-1}^m$ to get
$\O_k\subset\O_{k+1}$; for instance, when $m\in\mathcal{M}^+$ and $k$
is positive, the geodesics $A^m_k,B^m_k,\widetilde A^m_{k+1}$ do not
intersect $\widetilde A^m_k$, where $A^m_k$ (resp. $B^m_k; \widetilde
A^m_{k+1}; \widetilde A^m_k$) is defined as the geodesic arc joining
$p^m_{2k-1},p^m_{2k}$ (resp. $p^m_{2k},p^m_{2k+1}$; $p^m_{2k+1},
q^m_{2k+2}$; $p^m_{2k-1}, q^m_{2k}$).

\medskip

Let $\O_\infty$ be the semi-ideal polygonal domain with set of
vertices $\mathcal{V}_\infty^-\cup \mathcal{V}^0\cup
\mathcal{V}_\infty^+$, where
\[
\mathcal{V}_\infty^-=\{p_{-2k}^m,p_{1-2k}^m\ |\ m\in\mathcal{M}^-,\
k\in\N\} \cup \{q_{-2}^m\ |\ m\in\mathcal{M}-\mathcal{M}^-\},
\]
\[
\mathcal{V}_\infty^+=\{p_{2k-1}^m,p_{2k}^m\ |\ m\in\mathcal{M}^+,\
k\in\N\} \cup \{q_{2}^m\ |\ m\in\mathcal{M}-\mathcal{M}^+\}.
\]
It is clear that $\O_k\to\O_\infty$ as $k\to+\infty$. We can deduce,
arguing as above for $\O_k$, that $\O_\infty$ is a convex semi-ideal
polygonal domain verifying condition~$(\star)$ (the same condition
can be defined for the case of infinitely many vertices).

Since $p_{2k-1}\to p_\infty^m$ as $k\to +\infty$ when
$m\in\mathcal{M}^+$, and $p_{2k-1}\to p_\infty^{m-1}$ as $k\to
-\infty$ when $m\in\mathcal{M}^-$, then each $p^m_\infty$ is a limit
ideal vertex of $\O_\infty$ and it is:
\begin{itemize}
\item  left when $m\in{\cal M}^-$ and $m-1\not\in{\cal M}^+$;
\item  right when $m\not\in{\cal M}^-$ and $m-1\in{\cal M}^+$;
\item or 2-sided when $m\in{\cal M}^-$ and $m-1\in{\cal M}^+$.
\end{itemize}

\subsection{Construction of the Jenkins-Serrin minimal graphs}
Let $\Omega_k,\O_\infty$ be the domains constructed above.  We call
$A^m_i$ (resp. $B^m_i$) the geodesic arc joining $p^m_{2i-1},p^m_{2i}$
(resp.  $p^m_{2i},p^m_{2i+1}$), when they are defined; and $\widetilde
A^m_k$ (resp. $\widetilde B^m_k$; $\widetilde A^m_{-k}$; $\widetilde
B^m_{-k}$) the geodesic arc joining $p^m_{2k-1}, q^m_{2k}$
(resp. $q^{m}_{2k},p^{m+1}_\infty$; $p^m_\infty,q^m_{-2k}$;
$q^{m}_{-2k},p^m_{1-2k}$). 

By Theorem~\ref{th:JS}, there exists a solution $u_k$ (unique up to
an additive constant) for the minimal graph
equation~\eqref{eq.min.surf} in $\O_k$ with boundary values
$+\infty$ (resp. $-\infty$) on edges $A^m_i, \widetilde A^m_{k},
\widetilde A^m_{-k}$ (resp. $B^m_i, \widetilde B^m_{k}, \widetilde
B^m_{-k}$) which lie on $\partial\O_k$.

Fix a point $P\in\O_1$. We translate vertically the Jenkins-Serrin
graphs so that ${u_k(P)=0}$, for any $k$.

\begin{lemma}\label{lem:lines}
  The sequence $\{u_k\}_k$ has no divergence lines.
\end{lemma}
\begin{proof}
  Firstly, let us introduce some notation: For $\mu=2i-1$ or
  $\mu=\infty$, consider a sequence of nested horocycles $H^m_\mu(n)$
  at $p^m_\mu$ (in the case $p^m_\mu$ is defined) contained in
  $C^m_\mu$ such that $\mbox{dist}_\htwo(H^m_\mu(n),C^m_\mu)=n$ for
  any $n$.  In particular, the horocycles $H^m_{2i-1}(n)$ are pairwise
  disjoint for $n$ large.  Given a polygonal domain ${\cal P}_k$
  inscribed in $\O_k$, denote by ${\cal P}_k(n)$ the polygonal domain
  bounded by the part of $\partial{\cal P}_k$ outside the horocycles
  $H^m_{2i-1}(n),H^m_\infty(n)$ together with geodesic arcs joining
  points in $\partial{\cal P}_k\cap\left((\cup_{m,i}H^m_{2i-1}(n))\cup
    (\cup_mH^m_\infty(n))\right)$.  Also denote
  \[
  \a_k(n)=\sum_{m=1}^{m_0}\left( \sum_{i=1-k}^{k-1}
    |A^m_i\cap\partial{\cal P}_k(n)|+|\widetilde
    A^m_{-k}\cap\partial{\cal P}_k(n)|+|\widetilde
    A^m_k\cap\partial{\cal P}_k(n)|\right),
  \]
  \[
  \b_k(n)=\sum_{m=1}^{m_0}\left( \sum_{i=1-k}^{k-1}
    |B^m_i\cap\partial{\cal P}_k(n)|+|\widetilde
    B^m_{-k}\cap\partial{\cal P}_k(n)|+|\widetilde
    B^m_k\cap\partial{\cal P}_k(n)|\right),
  \]
  \[
  \ve_k(n)=|\partial{\cal P}_k(n)-\partial{\cal P}_k|.
  \]
  We observe that, for any fixed $k$, $\ve_k(n)\to 0$ as $n\to+\infty$.

  Now, let us prove Lemma~\ref{lem:lines}.  Suppose there exists a
  divergence line $L$ of $\{u_k\}_k$. As $\{\O_k\}_k$ is a monotone
  increasing sequence of domains converging to $\O_\infty$, then we
  can suppose $k$ is large enough so that $L\subset\O_k$.  We denote
  by $L(n)$ the geodesic arc in $L$ outside the horocycles
  $H^m_{2i-1}(n),H^m_\infty(n)$.  By Proposition~\ref{prop:div},
  $|F_{u_k}(L(n))|\to|L(n))|$ as $k\to+\infty$.

  We fix a component ${\cal P}_k$ of $\O_k-L$. By Lemma~
  \ref{lem:flux},
  \[
  |F_{u_k}(L(n))+\a_k(n)-\b_k(n)|\leq\ve_k(n),
  \]
  where $\a_k(n),\b_k(n),\ve_k(n)$ are defined as above for this
  choice of ${\cal P}_k$.
  \begin{itemize}
  \item In the case $L$ has finite length, we have $L(n)=L$ for $n$
    large enough. And $\a_k(n)-\b_k(n)=c$ is constant. Taking limits
    when $n$ goes to $+\infty$, we get $F_{u_k}(L)=-c$. This
    contradicts the fact that $|F_{u_k}(L)|<|L|$ but
    $|F_{u_k}(L)|\to|L|$ as $k\to+\infty$. Then $L$ must have infinite
    length.

  \item If $L$ joins either two ideal vertices
    $p^m_{2i-1},p^{m'}_{2j-1}$, two limit ideal vertices
    $p^m_\infty,p^{m'}_\infty$ or an ideal vertex $p^m_{2i-1}$ to a
    limit ideal vertex $p^{m'}_\infty$, then we have $\a_k(n)=\b_k(n)$
    because of the choice of horocycles above.  For any compact
    geodesic arc $T\subset L(n)$ and $k$ large, we have
    $|F_{u_k}(T)|\leq |F_{u_k}(L(n))|\leq \ve_k(n)$.  Taking
    $n\to+\infty$, we get $F_{u_k}(T)=0$. But this contradicts
    $|F_{u_k}(T)|\to|T|$ as $k\to+\infty$.

  \item Then $L$ must join a vertex $p^{m}_\mu$, with $\mu=2i-1$ or
    $\mu=\infty$, to a point $q$ in $\partial\Omega_k\cap\H^2$. Either
    $q$ is an interior vertex, and we denote it by $\widetilde q$,
    either it lies on an edge of $\Omega_k$ and we call $\widetilde q$
    the interior endpoint of such an edge. We can choose ${\cal P}_k$
    to have $\beta_k(n)\geq\alpha_k(n)$. Hence for $n$ large enough we
    have that $\beta_k(n)-\alpha_k(n)=n-c$, where
    $c=\mbox{dist}_\htwo(q,\widetilde q)$. Then
    \[
    F_{u_k}(L(n))=n-c-F_{u_k}(\partial{\cal P}_k(n)-\partial{\cal
      P}_k).
    \]
    (Observe that, in the case $L$ finishes at $p^{m}_\infty$, all the
    vertices $q_{2i}^m$ are contained in the same horocycle
    $C^{m}_\infty$.) Since $|L(n)|-n=d$ is constant for $n$ large, we
    get $|L(n)|-|F_{u_k}(L(n))|\to d+c$ as $n\to+\infty$.  On the
    other hand, $|F_{u_k}(L(n))|\to|L(n))|$ as $k\to+\infty$. So it
    must hold $c+d=0$. Therefore,
    \[
    \mbox{dist}_\htwo(\widetilde q,\partial H_\mu^m(n))\leq
    \mbox{dist}_\htwo(\widetilde q,q) + \mbox{dist}_\htwo(q,\partial
    H_\mu^m(n))=c+|L(n)|=n,
    \]
    which implies that $\widetilde q$ is contained in the horodisk
    bounded by $C_\mu^m$, in contradiction with the fact that
    $\Omega_k$ is a Jenkins-Serrin dommain (see Lemma~\ref{lem:car}).
\end{itemize}
\end{proof}

\begin{proposition}\label{prop:graph}
  Passing to a subsequence, $\{u_k\}_k$ converges uniformly on compact
  subsets of~$\O_\infty$ to a minimal graph $u_\infty$ such that it
  goes to $+\infty$ (resp. $-\infty$) as we approach within
  $\O_\infty$ to each~$A^m_i$ and each $\widetilde A^m_{-1}$
  (resp. each $B^m_i$ and eahc $\widetilde B^m_1$) in the boundary of
  $\Omega_\infty$.
\end{proposition}
\begin{proof}
  Since we have translated vertically the graphs $u_k$ so that
  ${u_k(P)=0}$ for any $k$, we get from Lemma~\ref{lem:lines} and
  Proposition~\ref{prop:div} that, after passing to a subsequence,
  $\{u_k\}_k$ converges to a minimal graph $u_\infty$, and the
  convergence is uniform on compact
  subsets of~$\O_\infty$. It is clear that $u_\infty(P)=0$.

  For any bounded geodesic arc $T\subset A^m_i$ we have
  $F_{u_k}(T)=|T|$ by Proposition~\ref{prop:div}; and then
  $F_{u_\infty}(T)=|T|$. Hence $u_\infty$ goes to $+\infty$ as we
  approach $T$ within~$\O_\infty$.  This proves
  $u_\infty|_{A^m_i}=+\infty$. Similarly we get $u_\infty|_{\widetilde
    A^m_{-1}}=+\infty$, $u_\infty|_{B^m_i}=-\infty$ and
  $u_\infty|_{\widetilde B^m_1}=-\infty$, which finishes
  Proposition~\ref{prop:graph}.
\end{proof}

\subsection{Passing to the conjugate surface}

Denote by $\Sigma_\infty$ (resp.  $\Sigma_k$) the graph surface of
$u_\infty$ (resp. $u_k$).  Observe that, if $m\in{\cal M}^+$ (resp.
$m\in{\cal M}^-$) and $i\geq 1$ (resp. $i\leq -1$), then the vertical
straight line $\Gamma^m_i=\{p^m_{2i}\}\times\R$ is contained in the
boundary of $\Sigma_\infty$ and of $\Sigma_k$, for any $k$ large; and
$\Gamma^m_0=\{p^m_0\}\times\R\subset \partial \Sigma_\infty
\cap \partial \Sigma_k$, for any $m$ and any $k$. We also denote
$\widetilde\Gamma^m_i=\{q^m_{2i}\}\times\R$. Then
$\widetilde\Gamma^m_k\subset\partial\Sigma_k$ and
$\widetilde\Gamma^m_{-1}\subset\partial\Sigma_\infty$ when $m\in{\cal
  M}^+$; and $\widetilde\Gamma^m_{-k}\subset\partial\Sigma_k$ and
$\widetilde\Gamma^m_1\subset\partial\Sigma_\infty$, when $m\in{\cal
  M}^-$.

We call $\Sigma_k^*$ the conjugate surface of $\Sigma_k$.  If $\Gamma$
is a curve in $\partial\Sigma_k$, then we denote by $\Gamma(k)^*$ the
corresponding curve in $\Sigma_k^*$.  We know (see
Subsection~\ref{subsec:semi-ideal}; or~\cite{moro1}, section~4) that
$\Sigma_k^*$ is a minimal graph bounded by horizontal geodesic
curvature lines contained in the same horizontal slice,
\[
\partial\Sigma_k^* =\cup_{m\in{\cal M}} \left( \Upsilon_k^{m-}\cup
\Gamma^m_0(k)^*\cup \Upsilon_k^{m+}\right) ,
\]
where
\[
\Upsilon_k^{m-} =\left\{
  \begin{array}{ll}
    \widetilde\Gamma^m_{-k}(k)^*\cup \Gamma^m_{1-k}(k)^*\cup
    \Gamma^m_{2-k}(k)^*\cup \ldots\cup\Gamma^m_{-1}(k)^*
    & \mbox{, if }\  m\in{\cal M}^-\\[3mm]
    \widetilde\Gamma^m_{-1}(k)^* & \mbox{, if }\  m\in{\cal M}-{\cal
      M}^-
  \end{array}
\right.
\]
\[
\Upsilon_k^{m+} =\left\{
  \begin{array}{ll}
    \Gamma^m_1(k)^*\cup\ldots\cup \Gamma^m_{k-2}(k)^*\cup
    \Gamma^m_{k-1}(k)^*\cup \widetilde\Gamma^m_{k}(k)^*
    & \mbox{, if }\  m\in{\cal M}^+\\[3mm]
    \widetilde\Gamma^m_{1}(k)^* & \mbox{, if }\  m\in{\cal M}-{\cal M}^+
  \end{array}
\right.
\]
Up to an isometry of $\htwo\times\R$, we can assume that the
horizontal geodesic curvature lines $\Gamma^m_i(k)^*$ are contained
in the horizontal slice $\htwo\times\{0\}$, and
$\Sigma_k^*\subset\{t\geq 0\}$. Each
$\Gamma^{m}_i(k)^*,\widetilde\Gamma^{m}_i(k)^* \subset\partial
\Sigma_k^*$ corresponds by conjugation, respectively, to
$\Gamma^m_i,\widetilde\Gamma^m_i\subset\partial\Sigma_k$.

If we denote by $\O_k^*$ the vertical projection of $\Sigma_k^*$
over $\htwo\equiv\htwo\times\{0\}$, then
\[
\partial\O_k^*=\cup_{m=1}^{m_0}\Big(\Lambda_k^{m-}\cup
\Gamma^{m}_0(k)^*\cup \eta^{m}_0(k)^*\cup \Lambda_k^{m+} \Big) ,
\]
cyclically ordered, where
\[
\Lambda_k^{m-} =\left\{
  \begin{array}{ll}
    \widetilde\Gamma^{m}_{-k}(k)^* \cup \eta^{m}_{-k}(k)^*\cup
    (\cup_{i=1-k}^{-1}\left(\Gamma^{m}_i(k)^* \cup
      \eta^{m}_i(k)^*\right))
    & \mbox{, if }\  m\in{\cal M}^-\\[3mm]
    \widetilde\Gamma^m_{-1}(k)^*\cup\eta^m_{-1}(k)^* & \mbox{, if }\
    m\in{\cal M}-{\cal M}^-
  \end{array}
\right.
\]
\[
\Lambda_k^{m+} =\left\{
  \begin{array}{ll}
    (\cup_{i=1}^{k-1}\left(\Gamma^{m}_i(k)^* \cup
      \eta^{m}_i(k)^*\right)) \cup \widetilde\Gamma^{m}_k(k)^* \cup
    \eta^{m}_k(k)^* & \mbox{, if }\  m\in{\cal M}^+\\[3mm]
    \widetilde\Gamma^m_{1}(k)^*\cup\eta^m_{1}(k)^* & \mbox{, if }\
    m\in{\cal M}-{\cal M}^+
  \end{array}
\right.
\]
and each $\eta^{m}_i(k)^*$ denotes a complete geodesic curve joining
at $\partial_\infty\htwo$ the corresponding curves in
$\partial\Sigma_k^*$. Furthermore, the curves $\Gamma^{m}_i(k)^*$,
$\widetilde\Gamma^{m}_{i}(k)^*$ are strictly concave with respect to
$\O_k^*$ (by the maximum principle).

Similarly, we denote by $\Sigma_\infty^*$ the conjugate surface of
$\Sigma_\infty$.

\begin{proposition}\label{prop:conj}
  $\Sigma_\infty^*\subset \{t\geq 0\}$ is a minimal graph over a
  domain $\O_\infty^*\subset\htwo$. Moreover,
  $\partial\Sigma_\infty^*\subset\htwo\times\{0\}$ consists of a
  collection of geodesic curvature lines,
  \[
  \partial\Sigma_\infty^*=\cup_{m=1}^{m_0}
  \Big(\Upsilon_\infty^{m-}\cup {\Gamma^{m}_0}^*\cup
  \Upsilon_\infty^{m+} \Big)
  \]
  (cyclically ordered), where
  \[
  \Upsilon_\infty^{m-} =\left\{
    \begin{array}{ll}
      \cup_{i=-\infty}^{-1}{\Gamma^{m}_i}^*
      & \mbox{, if }\  m\in{\cal M}^-\\[3mm]
      \widetilde{\Gamma^m_{-1}}^*
      & \mbox{, if }\  m\in{\cal M}-{\cal M}^-
    \end{array}
  \right.
  \]
  \[
  \Upsilon_\infty^{m+} =\left\{
    \begin{array}{ll}
      \cup_{i=1}^{+\infty}{\Gamma^{m}_i}^*
      & \mbox{, if }\  m\in{\cal M}^+\\[3mm]
      \widetilde{\Gamma^m_{1}}^*
      & \mbox{, if }\  m\in{\cal M}-{\cal M}^+
    \end{array}
  \right.
  \]
  Each ${\Gamma^{m}_i}^*$ is strictly concave with respect to
  $\Omega_\infty^*$. Moreover,
  \[
  \partial\O_\infty^*=\cup_{m=1}^{m_0}\Big(\Lambda_\infty^{m-}\cup
  \Gamma^{m}_0\cup {\eta^{m}_0}^*\cup \Lambda_\infty^{m+} \Big)
  \]
  (cyclically ordered), with
  \[
  \Lambda_\infty^{m-} =\left\{
    \begin{array}{ll}
      \cup_{i=-\infty}^{-1}\left({\Gamma^{m}_i}^* \cup
        {\eta^{m}_i}^*\right)
      & \mbox{, if }\  m\in{\cal M}^-\\[3mm]
      \widetilde{\Gamma^m_{-1}}^*\cup{\eta^m_{-1}}^* & \mbox{, if }\
      m\in{\cal M}-{\cal M}^-
    \end{array}
  \right.
  \]
  \[
  \Lambda_\infty^{m+} =\left\{
    \begin{array}{ll}
      \cup_{i=1}^{+\infty}\left({\Gamma^{m}_i}^* \cup
        {\eta^{m}_i}^*\right) & \mbox{, if }\  m\in{\cal M}^+\\[3mm]
      \widetilde{\Gamma^m_{1}}^*\cup{\eta^m_{1}}^* & \mbox{, if }\
      m\in{\cal M}-{\cal M}^+
    \end{array}
  \right.
  \]
  where ${\eta^{m}_i}^*$ denotes a complete geodesic curve asymptotic
  to its consecutive curves of $\partial\O_\infty^*$ at
  $\partial_\infty\htwo$.
\end{proposition}
\begin{proof}
  Theorem~\ref{th:krust} says $\Sigma_\infty^*$ is a minimal graph
  over certain domain $\O_\infty^*\subset\htwo$, because
  $\Sigma_\infty$ is a minimal graph over a convex domain. Moreover,
  since $\partial\Sigma_\infty=\cup_{m,i}\Gamma^m_i$ and each
  $\Gamma^{m}_i$ is a vertical geodesic curve, we get by
  Theorem~\ref{th:conjugate} that the boundary of $\Sigma_\infty^*$ is
  composed of horizontal geodesic curvature lines ${\Gamma^{m}_i}^*$.
  But we do not know a priori if they are all contained in the same
  horizontal slice.

  Let us prove that $\Sigma_\infty^*$ can be obtained as a limit of a
  subsequence of the conjugate graphs~$\Sigma_k^*$, when $k$ goes to
  $+\infty$ (in which case the curves
  $\Gamma^{m}_i(k)^*\subset\partial\Sigma_k^*$ converge to
  ${\Gamma^{m}_i}^*$).  This holds by~\cite[Proposition 2.10]{moro1},
  but we give the idea of the proof: If we prove that, after passing
  to a subsequence, the graphs $\Sigma_k^*$ converge to a surface $S$,
  then up to isometries of $\htwo\times\R$ we get $S=\Sigma_\infty^*$
  by Theorem~6 in~\cite{HST} (both $\Sigma_\infty^*,S$ are isometric
  to $\Sigma_\infty$; and the Hopf differentials associated to their
  vertical projection coincide with $-Q_\infty$, where $Q_\infty$ is
  the Hopf differential associated to the vertical projection of
  $\Sigma_\infty$). So we only have to obtain that the sequence
  $\{\Sigma_k^*\}$ converges. We know that the convergence domain
  associated to $\{u_k\}_k$ coincides with~$\O_\infty$.  Then, if we
  denote by $\nu_k$ the angle function of $\Sigma_k$, then
  $\{\nu_k\}_k$ is uniformly bounded away from zero on compact
  subsets. Since the angle function $\nu_k^*$ of $\Sigma_k^*$
  coincides with the one of $\Sigma_k$, then the same happens for
  $\{\nu_k^*\}_k$. We deduce from here that there are no divergence
  lines for $\{u_k^*\}_k$, and we get the convergence of the graphs
  $\Sigma_k^*$, passing to a subsequence.

  Since the graphs $\Sigma_k^*$ converge to $\Sigma_\infty^*$, then
  $\Sigma_\infty^*\subset \{t\geq 0\}$ and
  $\partial\Sigma_\infty^*\subset \{t=0\}$. We also deduce that the
  curves ${\Gamma^{m}_i}^*$ are cyclically ordered as follows:
  ${\Gamma^{m}_i}^*\leq {\Gamma^{m'}_j}^*$ if, and only if, $m<m'$ or
  $m=m'$ and $i\leq j$. By the maximum principle (using vertical
  geodesic planes), each ${\Gamma^{m}_i}^*\subset\partial\O_\infty^*$
  is strictly concave with respect to $\O_\infty^*$.

  Let us now prove that ${\Gamma^{m}_i}^*,{\Gamma^{m}_{i+1}}^*$ cannot
  finish at the same point $Q$ of $\partial_\infty\htwo$. Suppose this
  is the case. Since ${\Gamma^{m}_i}^*,{\Gamma^{m}_{i+1}}^*$ are
  strictly concave with respect to $\O_\infty^*$, we get
  $\mbox{dist}_{\htwo}({\Gamma^{m}_i}^*, {\Gamma^{m}_{i+1}}^*)=0$.
  Consider a triangle $T\subset\Omega_\infty^*$ bounded by subarcs of
  ${\Gamma^{m}_i}^*,{\Gamma^{m}_{i+1}}^*$ and a geodesic arc $c'$
  joining points in ${\Gamma^{m}_i}^*,{\Gamma^{m}_{i+1}}^*$.  Let
  $u_\infty^*:T\to\R$ define the graph $\Sigma_\infty^*$ over $T$.
  Then $u_\infty^*$ has boundary values $0$ on
  ${\Gamma^{m}_i}^*,{\Gamma^{m}_{i+1}}^*$ and a bounded continuous
  function over~$c'$. We call $c$ the complete geodesic of $\htwo$
  containing $c'$ and we consider the minimal graph $w^+$
  (resp. $w^-$) over the component $\Delta$ of $\htwo-c$ which
  contains $T$, which has boundary values $+\infty$ (resp. $-\infty$)
  over $c$ and $0$ over $\partial\Delta\cap\partial_\infty\htwo$. By
  the maximum principle, $w^-|_T\leq u_\infty^*\leq w^+|_T$. Hence we
  deduce that $u_\infty^*$ converges to $0$ as we approach $Q$ in any
  direction, and then
  $\mbox{dist}_{\Sigma_\infty^*}({\Gamma^{m}_i}^*,{\Gamma^{m}_{i+1}}^*)=0$.
  But $\Sigma_\infty,\Sigma_\infty^*$ are isometric and
  $\mbox{dist}_{\Sigma_\infty}({\Gamma^{m}_i},{\Gamma^{m}_{i+1}})\geq
  \mbox{dist}_{\htwo}({p^{m}_{2i}},{p^{m}_{2i+2}})>0$, a
  contradiction.

  Therefore, the geodesics $\eta^{m}_i(k)^*$ in the boundary of
  $\Omega_k^*$ converge to a geodesic
  ${\eta^{m}_i}^*\subset\partial\Omega_\infty^*$ over which
  $\Sigma_\infty^*$ goes to $+\infty$. Thus
  $\partial\O_\infty^*=\cup_{m=1}^{m_0}
  \Big(\cup_{i=-\infty}^{+\infty}\left(\Gamma^{m\, *}_i \cup \eta^{m\,
      *}_i\right)\Big)$, cyclically ordered. This finishes the proof
  of Proposition~\ref{prop:conj}.
\end{proof}

If we reflect $\Sigma_\infty^*$ with respect to $\htwo\times\{0\}$,
we get a properly embedded minimal surface~$M$ of genus zero and
infinitely many planar ends in $\htwo\times\R$.  The non-limit ends
of $M$ are asymptotic to the vertical geodesic planes $\eta^{m\,
*}_i\times\R$. We can deduce that there is exactly one limit end
from $\eta^{m\, *}_0\times\R$ to $\eta^{m+1\,*}_0\times\R$, that we
call $E^m_\infty$; and $E^m_\infty$ is a left (resp. right, 2-sided)
limit end when $p^m_\infty$ is  a left (resp. right, 2-sided) limit
ideal vertex.

\section{Proof of Theorem~\ref{th:finite}: infinite countable case}
\label{sec:countable} In this section we construct properly embedded
minimal surfaces in $\htwo\times\R$ with genus zero, infinitely many
vertical planar ends and an infinite countable number of limit ends
$\{E^m_\infty\ |\ m\in\N\}$. Furthermore, as in the finite case, we
can prescribe if each limit end $E^k_\infty$ is left, right or
2-sided.

We follow the same sketch as in section~\ref{sec:finite}. We firstly
construct, by taking limits of a monotone increasing sequence of
convex Jenkins-Serrin semi-ideal polygonal domains $\O_k$ with
finitely many vertices and satisfying condition $(\star)$, a convex
semi-ideal polygonal domain $\O_\infty$ with an infinite countable
number of limit ideal vertices $\{p_\infty^m\ |\ m\in\N\}$ such
that, if $E_\infty^m$ is prescribed to be a left (resp. a right or a
2-sided) limit end, then $p_\infty^m$ is a left (resp. a right or a
2-sided) limit ideal vertex. The remaining part of the construction
follows exactly as in Section~\ref{sec:finite}, replacing $m_0$ by
$+\infty$ and ${\cal M}$ by $\N$.

\subsection{Construction of the domains}

Consider $p_\infty^1=\infty$ and two ideal points
$p_\infty^2=(x_\infty^2,0)$ and $p_\infty^3=(x_\infty^3,0)$, with
$-1<x_\infty^2<x_\infty^3\leq 1$. These points will be limit ideal
vertices of $\O_\infty$. We call $C^1_\infty=\{y=1\}$ and
$C_\infty^2$ (resp. $C_\infty^3$) the horocycle at $p_\infty^2$
(resp. $p_\infty^3$) passing through $P_0=(0,1)$.

\begin{figure}
\begin{center}
\includegraphics[scale=1]{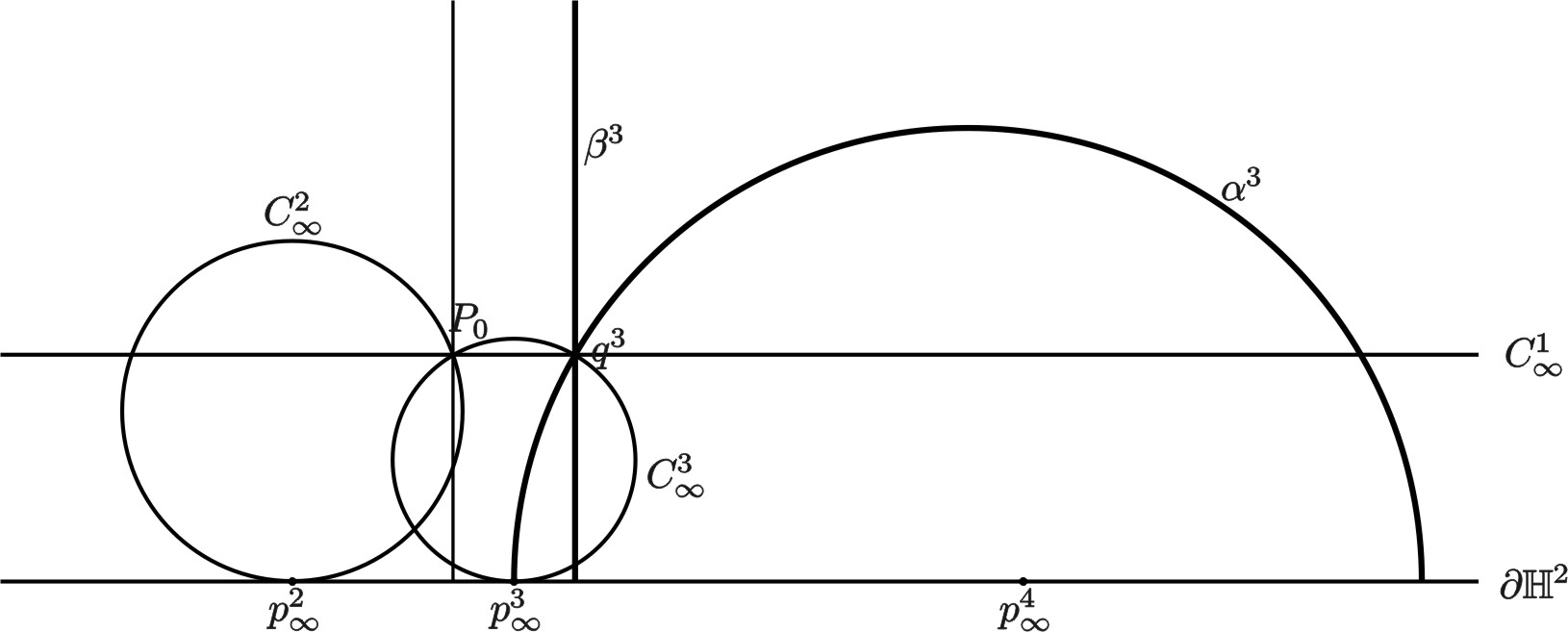}
\caption{Choice of $p_\infty^4$.} \label{fig:contable1}
\end{center}
\end{figure}

Let $q^3$ be the point in $C_\infty^3\cap C_\infty^1$ different
from~$P_0$, $\a^3$ be the complete geodesic curve with endpoint
$p_\infty^3$ passing through $q^3$ and $\b^3=\{x=b^3\}$, where $b^3$
is the constant for which $\b^3$ passes through $q^3$. We take any
point $p_\infty^4=(x_\infty^4,0)$ with $b^3\leq x_\infty^4\leq a^3$,
where $(a^3,0)$ is the endpoint of $\a^3$ different
from~$p_\infty^3$, see Figure~\ref{fig:contable1}.

Let us now define by induction the remaining limit ideal vertices
$p_\infty^k$, $k\in\N$, of $\O_\infty$. Assume we have define
$p_\infty^4=(x_\infty^4,0),\ldots, p_\infty^k=(x_\infty^k,0)$, with
$x_\infty^3<x_\infty^4<\ldots<x_\infty^k<+\infty$. For any $4\leq
i\leq k$, let $C_\infty^i$ be the horocycle at $p_\infty^i$ passing
through $P_0=(0,1)$, and $q^i$ be the point in $C_\infty^i\cap
C_\infty^1$ different from $P_0$. We also consider the complete
geodesic curve $\a^i$ with endpoint $p_\infty^i$ passing through
$q^i$ and $\b^i=\{x=b^i\}$ the geodesic which contains $q^i$. Denote
by $(a^i,0)$ the endpoint of $\a^i$ different from~$p_\infty^i$. We
assume that $b^{i-1}\leq x_\infty^i\leq a^{i-1}$. We now define
$p_\infty^{k+1}$ with the same property: We take any point
$p_\infty^{k+1}=(x_\infty^{k+1},0)$ such that $b^{k}\leq
x_\infty^{k+1}\leq a^{k}$.

We now want to define the non-limit ideal vertices of $\O_\infty$.
We consider:
\[
\mathcal{M}^+=\{ m\in\N\ |\ E_\infty^{m+1}\ \mbox{ is prescribed to
be either a right or a 2-sided limit end}\},
\]
\[
\mathcal{M}^-=\{ m\in\N\ |\ E_\infty^{m}\ \mbox{ is prescribed to be
either a left or a 2-sided limit end}\}.
\]
For any $m\in\N$, we define exactly as in
Subsection~\ref{subsec:finite} (using this new definition of the
sets $\mathcal{M}^+, \mathcal{M}^-$) the ideal vertices $p_{2i-1}^m$
of $\O_\infty$ placed from $p_\infty^m$ to $p_\infty^{m+1}$, the
horocycles $C_{2i-1}^m$, the interior vertices $p_{2i}^m$ and the
interior points $q_{2i}^m$.

We can now define the monotone sequence of semi-ideal Jenkins-Serrin
polygonal domains $\O_k$: We call $\O_1$ the semi-ideal polygonal
domain with vertices
\[
\{p_\infty^m, q_{-2}^m, p_{-1}^m, p_0^m, p_1^m, q_2^m\ |\ 1\leq
m\leq 2\} \cup\{p_\infty^3,\, q^3\}.
\]
For $k\geq 2$, let $\O_k$ be defined as the semi-ideal polygonal
domain with vertices
\[
\{p_\infty^m,p_{-1}^m,p_0^m,p_1^m\ |\ 1\leq m\leq k+1\} \cup
\mathcal{V}_k^- \cup \mathcal{V}_k^+ \cup\{p_\infty^{k+2},\,
q^{k+2}\} ,
\]
where
\[
\mathcal{V}_k^-=\{q_{-2k}^m,p_{1-2k}^m,p_{2-2k}^m,\ldots,
p_{-3}^m,p_{-2}^m\ |\ 1\leq m\leq k+1,\ m\in\mathcal{M}^-\}
\]
\[
\cup
\{q_{-2}^m\ |\ 1\leq m\leq k+1,\ m\in\mathcal{M}-\mathcal{M}^-\},
\]
\[
\mathcal{V}_k^+=\{p_{2}^m,p_{3}^m,\ldots,
p_{2k-2}^m,p_{2k-1}^m,q_{2k}^m\ |\ 1\leq m\leq k+1,\
m\in\mathcal{M}^+\}
\]
\[
\cup \{q_{2}^m\ |\ 1\leq m\leq k+1,\
 m\in\mathcal{M}-\mathcal{M}^+\}.
\]

\begin{figure}
\begin{center}
\includegraphics[scale=1.1]{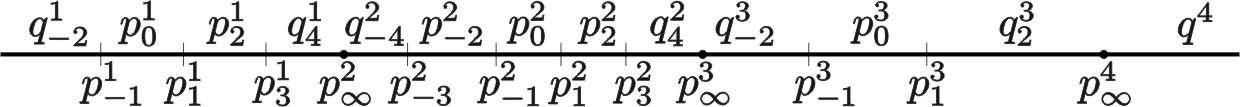}
\caption{Sketch of the vertices of $\O_2$ in the case $p_\infty^1$
and $p_\infty^3$ are right limit ideal vertices, $p_\infty^2$ is a
2-sided limit ideal vertex and $p_\infty^4$ is a left limit ideal
vertex .} \label{fig:contable2}
\end{center}
\end{figure}

The domain $\O_k$ has $4(k+1)+2+N$ vertices, where $2(k+1)\leq N\leq
2(k+1)(2k-1)$ depends on the number of left, right or 2-sided limit
ideal ends in $\{E_\infty^1,\ldots,E_\infty^{k+1}\}$. As in
Subsection~\ref{subsec:finite}, $\O_k$ is a convex Jenkins-Serrin
semi-ideal polygonal domain satisfying condition $(\star)$, and
$\O_k\subset\O_{k+1}$. When $k$ goes to $+\infty$, $\O_k$ converges
to the convex semi-ideal polygonal domain $\O_\infty$ with set of
vertices $\mathcal{V}_\infty^-\cup \mathcal{V}^0\cup
\mathcal{V}_\infty^+$, where
\[
\mathcal{V}_\infty^-=\{p_{-2k}^m,p_{1-2k}^m\ |\ k\in\N,\
m\in\mathcal{M}^-\} \cup \{q_{-2}^m\ |\
m\in\mathcal{M}-\mathcal{M}^-\},
\]
\[
\mathcal{V}_\infty^0=\{p_\infty^m,p_{-1}^m,p_0^m,p_1^m\ |\ m\in\N\}
\]
\[
\mathcal{V}_\infty^+=\{p_{2k-1}^m,p_{2k}^m\ |\ k\in\N,\
m\in\mathcal{M}^+\} \cup \{q_{2}^m\ |\
m\in\mathcal{M}-\mathcal{M}^+\}.
\]

\bibliographystyle{plain}

\noindent
M. Magdalena Rodr\'\i guez\\
Departamento de Geometr\'\i a y Topolog\'\i a\\
Universidad de Granada\\
Campus de Fuentenueva, s/n\\
18071, Granada, Spain\\
e-mail: \texttt{magdarp@ugr.es}

\end{document}